\documentclass[11pt,leqno]{amsart}
\usepackage{epsfig}
\usepackage{amssymb}
\usepackage{amscd}
\usepackage[matrix,arrow]{xy}
\usepackage{graphicx}

\newtheorem{theo}{Theorem}[section]
\newtheorem{lemma}[theo]{Lemma}

\newtheorem{cor}[theo]{Corollary}

\numberwithin{equation}{section}

\def\A{{\mathbb A}}

\def\C{\mathbb{C}}
\def\Z{\mathbb{Z}}
\def\Q{\mathbb{Q}}

\def\pre-tr{\operatorname{pre-tr}}

\def\Hom{\operatorname{Hom}}

\newcommand{\cH}{{\mathcal H}}

\newcommand{\mrB}{\operatorname{B}}
\newcommand{\mrE}{\operatorname{E}}

\newcommand{\supp}{\operatorname{Supp}}

\newcommand{\Ext}{\operatorname{Ext}}

\newcommand{\Gr}{\operatorname{Gr}}
\newcommand{\Fl}{\operatorname{Fl}}

\usepackage{epsf}
\usepackage{amscd}

\title[Cohomological Hall algebra of a symmetric quiver]
{Cohomological Hall algebra of a symmetric quiver}

\author{Alexander I. Efimov}
\address{Steklov Mathematical Institute of RAS, Gubkin str. 8, GSP-1, Moscow 119991, Russia}
\address{Independent University of Moscow, Moscow,
Russia} \email{efimov@mccme.ru}
\thanks{The author was partially supported by
"Dynasty" Foundation, RFBR (grant 4713.2010.1), and by AG Laboratory HSE, RF government grant, ag. 11.G34.31.0023.}

\begin{document}

\begin{abstract}In the paper \cite{KS}, Kontsevich and Soibelman in particular associate to each finite quiver $Q$ with a set of vertices $I$ the so-called Cohomological Hall algebra $\cH,$ which is $\Z_{\geq 0}^I$-graded. Its graded component $\cH_{\gamma}$ is defined as cohomology of Artin moduli stack
 of representations  with dimension vector $\gamma.$ The product comes from natural correspondences which parameterize extensions of representations.

In the case of symmetric quiver, one can refine the grading to $\Z_{\geq 0}^I\times\Z,$ and modify the product by a sign to get a super-commutative
algebra $(\cH,\star)$ (with parity induced by $\Z$-grading). It is conjectured in \cite{KS} that in this case the algebra $(\cH\otimes\Q,\star)$
 is free super-commutative generated by a $\Z_{\geq 0}^I\times\Z$-graded vector space of the form $V=V^{prim}\otimes\Q[x],$
  where $x$ is a variable of bidegree $(0,2)\in\Z_{\geq 0}^I\times\Z,$ and all the spaces $\bigoplus\limits_{k\in\Z}V^{prim}_{\gamma,k},$ $\gamma\in\Z_{\geq 0}^I.$
are finite-dimensional. In this paper we prove this conjecture (Theorem \ref{free_algebra_intro}).

   We also prove some explicit bounds on pairs $(\gamma,k)$ for which $V^{prim}_{\gamma,k}\ne 0$ (Theorem \ref{upper_bound_intro}).
   Passing to generating functions, we obtain the positivity result for quantum Donaldson-Thomas invariants, which was used 
   by S. Mozgovoy to prove Kac's conjecture for quivers with sufficiently many loops \cite{M}.
    Finally, we mention a connection with the paper of Reineke \cite{R}.\end{abstract}

\maketitle

\tableofcontents

\section{Introduction}

In this paper we study Cohomological Hall algebra (COHA) introduced by Kontsevich and Soibelman \cite{KS}, in the case of symmetric quiver
without potential. Our main result is the proof of Kontsevich-Soibelman conjecture on the freeness of COHA of symmetric quiver.

Consider a finite quiver $Q$ with a set of vertices $I$ and with $a_{ij}$ edges from $i\in I$ to $j\in I,$ so that $a_{ij}\in\Z_{\geq 0}.$
One can choose a very degenerate stability condition on the category of complex finite-dimensional representations, so that
stable representations are precisely the simple ones, and they all have the same slope. In particular, each representation is semi-stable with the same slope. Then, for each dimension vector $$\gamma=\{\gamma^i\}_{i\in I}\in\Z_{\geq 0}^I,$$ the moduli space of representations of $Q$  is
an Artin quotient stack $M_{\gamma}/G_{\gamma},$ where $M_{\gamma}$ is an affine space of all representations in coordinate vector spaces $\C^{\gamma^i},$
$G_{\gamma}=\prod\limits_{i\in I}GL(\gamma^i,\C),$ and the action is by conjugation. One then defines a $\Z_{\geq 0}^I$-graded $\Q$-vector space
$\cH$ by the formula
$$\cH=\bigoplus\limits_{\gamma\in \Z_{\geq 0}^I}\cH_{\gamma},\quad \cH_{\gamma}:=H^{\cdot}_{G_{\gamma}}(M_{\gamma,\Q}).$$
Note that originally in \cite{KS}, one takes cohomology with integer coefficients, but we will deal only with the result of tensoring by $\Q.$

Now, for each two vectors $\gamma_1,\gamma_2\in\Z_{\geq 0}^I,$ one has natural correspondence between the stacks $M_{\gamma_1}/G_{\gamma_1}$
and $M_{\gamma_2}/G_{\gamma_2},$ which parameterizes all extensions. We get natural maps of stacks
$$(M_{\gamma_1}/G_{\gamma_1})\times (M_{\gamma_2}/G_{\gamma_2})\leftarrow M_{\gamma_1,\gamma_2}/G_{\gamma_1,\gamma_2}\to M_{\gamma_1+\gamma_2}/G_{\gamma_1+\gamma_2},$$
which allow one to define the multiplication
\begin{equation}\label{shift}H^{\cdot}_{G_{\gamma_1}}(M_{\gamma_1})\otimes H^{\cdot}_{G_{\gamma_2}}(M_{\gamma_2})\to H^{\cdot-2\chi_Q(\gamma_1,\gamma_2)}_{G_{\gamma_1+\gamma_2}}(M_{\gamma_1+\gamma_2}),\end{equation}
where $\chi_Q(\gamma_1,\gamma_2)$ is the Euler form:
$$\chi_Q(\gamma_1,\gamma_2)=\sum\limits_{i\in I}\gamma_1^i\gamma_2^i-\sum\limits_{i,j\in I}a_{ij}\gamma_1^i\gamma_2^j.$$

It is proved in \cite{KS}, Theorem 1, that the resulting product on $\cH$ is associative, so this makes $\cH$ into a $\Z_{\geq 0}^I$-graded algebra,
which is called a (rational) Cohomological Hall algebra of a quiver $Q.$

Now we restrict to the case of symmetric quiver $Q,$ i.e. to the case $a_{ij}=a_{ji}.$ In this case the Euler form $\chi_Q(\gamma_1,\gamma_2)$
is symmetric as well. One defines a $(\Z_{\geq 0}^I\times\Z)$-graded algebra structure on $\cH,$ by assigning to a subspace $H^{k}_{G_{\gamma}}(M_{\gamma})$
a bigrading $(\gamma,k+\chi_Q(\gamma,\gamma)).$ It follows from the formula \eqref{shift} that the product is compatible with this grading. We also define a parity on $\cH$ to be induced by $\Z$-grading.

In general, the algebra $\cH$ for symmetric quiver is not super-commutative, but it becomes such after twisting the product by sign. Denote by $\star$ the resulting
super-commutative product. Our main result is the following Theorem which was conjectured in \cite{KS} (Conjecture 1).

\begin{theo}\label{free_algebra_intro}For any finite symmetric quiver $Q,$ the $(\Z_{\geq 0}^I\times\Z)$-graded algebra $(\cH,\star)$
is a free super-commutative algebra generated by a $(\Z_{\geq 0}^I\times\Z)$-graded vector space $V$
of the form $V=V^{prim}\otimes\Q[x],$ where $x$ is a variable of degree $(0,2)\in\Z_{\geq 0}^I\times\Z,$
and for any $\gamma\in\Z_{\geq 0}^I$ the space $V_{\gamma,k}^{prim}$ is non-zero (and finite-dimensional) only for finitely many $k\in\Z.$\end{theo}

The second result in this paper gives explicit bounds on pairs $(\gamma,k)$ for which $V_{\gamma,k}^{prim}\ne 0.$
For a given symmetric quiver $Q,$ and $\gamma\in \Z_{\geq 0}^I\setminus\{0\},$ we put
$$N_{\gamma}(Q):=\frac12(\sum\limits_{\substack{i,j\in I,\\ i\ne j}}a_{ij}\gamma^i\gamma^j+\sum\limits_{i\in I}\max(a_{ii}-1,0)\gamma^i(\gamma^i-1))-\sum\limits_{i\in I}\gamma^i+2.$$

\begin{theo}\label{upper_bound_intro}In the notation of Theorem \ref{free_algebra_intro}, if $V_{\gamma,k}^{prim}\ne 0,$ then $\gamma\ne 0,$
$$k\equiv \chi_Q(\gamma,\gamma)\text{ mod }2,\quad\text{and}\quad \chi_Q(\gamma,\gamma)\leq k< \chi_Q(\gamma,\gamma)+2N_{\gamma}(Q).$$\end{theo}

The only non-trivial statement in Theorem \ref{upper_bound_intro} is the upper bound on $k.$ In the proofs of both Theorems, we use explicit formulas for the product in $\cH$ from \cite{KS}, Theorem 2. Namely, since the affine space $M_{\gamma}$ is $G_{\gamma}$-equivariantly
contractible, we have
$$\cH_{\gamma}\cong H^{\cdot}(\mrB G_{\gamma}),$$
and the RHS is isomorphic to the algebra of polynomials in $x_{i,\alpha},$ where $i\in I,$ $1\leq \alpha\leq \gamma^i,$
which are invariant with respect to the product of symmetric groups $S_{\gamma^i}.$ Then, given two polynomials $f_1\in \cH_{\gamma_1},$ $f_2\in\cH_{\gamma_2},$ their product $f_1\cdot f_2\in \cH_{\gamma},$ $\gamma=\gamma_1+\gamma_2,$ equals to the sum over all shuffles (for any $i\in I$) of the following rational function in variables $(x_{i,\alpha}')_{i\in I,\alpha\in \{1,\dots,\gamma_1^i\}},$
$(x_{i,\alpha}'')_{i\in I,\alpha\in \{1,\dots,\gamma_2^i\}}:$
$$f_1((x_{i,\alpha}'))f_2((x_{i,\alpha}''))
\frac{\prod\limits_{i,j\in I}\prod\limits_{\alpha_1=1}^{\gamma_1^i}\prod\limits_{\alpha_2=1}^{\gamma_2^j}(x_{j,\alpha_2}''-x_{i,\alpha_1}')^{a_{ij}}}
{\prod\limits_{i\in I}\prod\limits_{\alpha_1=1}^{\gamma_1^i}\prod\limits_{\alpha_2=1}^{\gamma_2^i}(x_{i,\alpha_2}''-x_{i,\alpha_1}')}.$$

Theorems \ref{free_algebra_intro} and \ref{upper_bound_intro} imply the corresponding results for the generating functions for Cohomological Hall algebras,
in particular, positivity for quantum Donaldson-Thomas invariants. The positivity result was used by S. Mozgovoy to prove
Kaz's conjecture for quivers with at least one loop at each vertex \cite{M}

The paper is organized as follows.

Section \ref{s:preliminaries} is devoted to some preliminaries on Cohomological Hall algebras for quivers. We follow \cite{KS}, Section 2.
In Subsection \ref{ss:COHA_definition} we give a definition of rational Cohomological Hall algebra for an arbitrary finite quiver.
Subsection \ref{ss:explicit} is devoted to explicit formulas for the product in Cohomological Hall algebras.
In Subsection \ref{ss:grading} we define an additional $\Z$-grading on COHA of symmetric quiver, so that we get a $(\Z_{\geq 0}^I\times\Z)$-graded
algebra. Then, we show how to modify a product on $\cH$ by a sign to get a super-commutative algebra $(\cH,\star),$ with parity induced
by $\Z$-grading.

Section \ref{s:main_results} is devoted to the proofs of Theorem \ref{free_algebra_intro} (Theorem \ref{free_algebra}) and
Theorem \ref{upper_bound_intro} (Theorem \ref{upper_bound}).

In Section \ref{s:DT_invariants} we discuss applications of our results to the generating function of COHA, or, in other words, to quantized Donaldson-Thomas invariants.

\smallskip

{\noindent {\bf Acknowledgements.}} I am grateful to Maxim Kontsevich and Yan Soibelman for useful discussions. I am also
grateful to Sergey Mozgovoy for pointing my attention to his paper on Kac's conjecture \cite{M}.

\section{Preliminaries on Cohomological Hall algebras}
\label{s:preliminaries}

In this Section we recall some definitions and results from \cite{KS}, Section 2.

\subsection{COHA of a quiver}
\label{ss:COHA_definition}

Let $Q$ be a finite quiver. Denote its set of vertices by $I,$ and let $a_{ij}\in\Z_{\geq 0}$ be the number of arrows from $i$ to $j,$
where $i,j\in I.$ Fix a dimension vector $\gamma=(\gamma^i)_{i\in I}\in\Z_{\geq 0}^I.$ We have an affine variety of representations of $Q$
in complex coordinate vector spaces $\C^{\gamma^i}:$
$$M_{\gamma}=\prod\limits_{i,j\in I}\C^{a_{ij}\gamma^i\gamma^j}.$$
The variety $M_{\gamma}$ is acted via conjugation by the complex algebraic group
$G_{\gamma}=\prod\limits_{i\in I}GL(\gamma^i,\C).$

Recall that infinite-dimensional Grammanian
$$\Gr(d,\infty)=\lim\limits_{\to}Gr(d,\C^n),\quad n\to+\infty,$$
is a model for the classifying space of $GL(d,\C).$ Put
$$\mrB G_{\gamma}:=\prod\limits_{i\in I}\mrB GL(\gamma^i,\C)=\prod\limits_{i\in I}\Gr(\gamma^i,\infty).$$
We have a standard universal $G_{\gamma}$-bundle $\mrE G_{\gamma}\to\mrB G_{\gamma},$ and the Artin stack $M_{\gamma}/G_{\gamma}$ gives a universal
family over $\mrB G_{\gamma}:$
$$M_{\gamma}^{univ}:=(\mrE G_{\gamma}\times M_{\gamma})/G_{\gamma}\to \mrE G_{\gamma}/G_{\gamma}=\mrB G_{\gamma}.$$

Define a $\Z_{\geq 0}^I$-graded $\Q$-vector space
$$\cH=\bigoplus_{\gamma\in\Z_{\geq 0}^I}\cH_{\gamma},$$
putting
$$\cH_{\gamma}:=H^{\cdot}_{G_{\gamma}}(M_{\gamma},\Q)=\bigoplus_{n\geq 0}H^n(M_{\gamma}^{univ},\Q).$$

Now we define a multiplication on $\cH$ which makes it into associative unital $\Z_{\geq 0}^I$-graded algebra over $\Q.$
Take two vectors $\gamma_1,\gamma_2\in \Z_{\geq 0}^I,$ and put $\gamma:=\gamma_1+\gamma_2.$ Consider the affine subspace $M_{\gamma_1,\gamma_2}\subset M_{\gamma},$ which consists of representations for which standard subspaces $\C^{\gamma_1^i}\subset \C^{\gamma^i}$ form a
subrepresentation. The subspace $M_{\gamma_1,\gamma_2}$ is preserved by the action of the subgroup $G_{\gamma_1,\gamma_2}\subset G_{\gamma}$
which consists of transformations preserving subspaces $\C^{\gamma_1^i}\subset\C^{\gamma^i}.$ We use a model for $\mrB G_{\gamma_1,\gamma_2}$
which is the total space of a bundle over $\mrB G_{\gamma}$ with fiber $G_{\gamma}/G_{\gamma_1,\gamma_2}$ (i.e. a product of infinite-dimensional partial flag varieties $\Fl(\gamma_1^i,\gamma_i,\infty)$). We have natural projection $\mrE G_{\gamma}\to \mrB G_{\gamma_1,\gamma_2}$ which is a universal
$G_{\gamma_1,\gamma_2}$-bundle.

Now define the morphism
$$m_{\gamma_1,\gamma_2}:\cH_{\gamma_1}\otimes\cH_{\gamma_2}\to \cH_{\gamma}$$
as the composition of the K\"unneth isomorphism
$$\otimes:H_{G_{\gamma_1}}^{\cdot}(M_{\gamma_1,\Q})\otimes H_{G_{\gamma_2}}^{\cdot}(M_2,\Q)\stackrel{\cong}{\to} H_{G_{\gamma_1\times G_{\gamma_2}}}^{\cdot}(M_{\gamma_1}\times M_{\gamma_2},\Q)$$
and the following morphisms:
$$H_{G_{\gamma_1\times G_{\gamma_2}}}^{\cdot}(M_{\gamma_1}\times M_{\gamma_2},\Q)\stackrel{\cong}{\to}H_{G_{\gamma_1,G_{\gamma_2}}}^{\cdot}
(M_{\gamma_1,\gamma_2},\Q)\to H^{\cdot+2c_1}_{G_{\gamma_1,\gamma_2}}(M_{\gamma},\Q)\to H^{\cdot+2c_1+2c_2}_{G_{\gamma}}(M_{\gamma}).$$
Here the first map is induced by natural surjective homotopy equivalences $$M_{\gamma_1,\gamma_2}\stackrel{\sim}{\to}M_{\gamma_1}\times M_{\gamma_2}, G_{\gamma_1,\gamma_2}\to G_{\gamma_1}\times G_{\gamma_2}.$$
The other two maps are natural pushforward morphisms, with
$$c_1=\dim_{\C} M_{\gamma}-\dim_{\C} M_{\gamma_1,\gamma_2},\quad c_2=\dim_{\C}G_{\gamma_1,\gamma_2}-\dim_{\C}G_{\gamma}.$$

\begin{theo}(\cite{KS}, Theorem 1) The constructed product $m$ on $\cH$ is associative.\end{theo}

Note that \begin{equation}\label{Euler}c_1+c_2=-\chi_Q(\gamma_1,\gamma_2),\end{equation} where
$$\chi_Q(\gamma_1,\gamma_2)=\sum\limits_{i\in I}\gamma_1^i\gamma_2^i-\sum\limits_{i,j\in I}a_{ij}\gamma_1^i\gamma_2^j$$
is the Euler form of the quiver $Q.$ That is, given two representations $R_1,R_2$ (over any field) of the quiver $Q,$ with dimension vectors $\gamma_1,\gamma_2$ respectively, one has
$$\sum\limits_i(-1)^i\dim\Ext^i(R_1,R_2)=\dim\Hom(R_1,R_2)-\dim\Ext^1(R_1,R_2)=\chi_Q(\gamma_1,\gamma_2).$$

\subsection{Explicit description of COHA of a quiver}
\label{ss:explicit}

Since the affine spaces $M_{\gamma}$ are $G_{\gamma}$-equivariantly contractible, we have natural isomorphisms
$$\cH_{\gamma}\cong H^{\cdot}(\mrB G_{\gamma},\Q)=\bigotimes_{i\in I}H^{\cdot}(\mrB GL(\gamma^i,\C),\Q).$$
Recall that
$$H^{\cdot}(\mrB GL(d,\C),\Q)\cong \Q[x_1,\dots,x_d]^{S_d}.$$
For a vector $\gamma\in\Z_{\geq 0}^{I},$ introduce variables $x_{i,\alpha},$ where $i\in I,$ $\alpha\in \{1,\dots,\gamma^i\}.$
Then, we get natural isomorphisms
$$\cH_{\gamma}\cong\Q[\{x_{i,\alpha}\}_{i\in I, \alpha\in \{1,\dots,\gamma^i\}}]^{\prod\limits_{i\in I}S_{\gamma^i}}.$$
From this moment, we identify the elements of $\cH_{\gamma}$ with the corresponding polynomials.

\begin{theo}\label{explicit}(\cite{KS}, Theorem 2)
Given two polynomials $f_1\in \cH_{\gamma_1},$ $f_2\in\cH_{\gamma_2},$ their product $f_1\cdot f_2\in \cH_{\gamma},$ $\gamma=\gamma_1+\gamma_2,$ equals to the sum over all shuffles (for any $i\in I$) of the following rational function in variables $(x_{i,\alpha}')_{i\in I,\alpha\in \{1,\dots,\gamma_1^i\}},$
$(x_{i,\alpha}'')_{i\in I,\alpha\in \{1,\dots,\gamma_2^i\}}:$
$$f_1((x_{i,\alpha}'))f_2((x_{i,\alpha}''))
\frac{\prod\limits_{i,j\in I}\prod\limits_{\alpha_1=1}^{\gamma_1^i}\prod\limits_{\alpha_2=1}^{\gamma_2^j}(x_{j,\alpha_2}''-x_{i,\alpha_1}')^{a_{ij}}}
{\prod\limits_{i\in I}\prod\limits_{\alpha_1=1}^{\gamma_1^i}\prod\limits_{\alpha_2=1}^{\gamma_2^i}(x_{i,\alpha_2}''-x_{i,\alpha_1}')}.$$
\end{theo}

\subsection{Additional grading in the symmetric case}
\label{ss:grading}

Now assume that the quiver $Q$ is symmetric, i.e. $a_{ij}=a_{ji},$ $i,j\in I.$ Then the Euler form
$$\chi_Q(\gamma_1,\gamma_2)=\sum\limits_{i\in I}\gamma_1^i\gamma_2^i-\sum\limits_{i,j\in I}a_{ij}\gamma_1^i\gamma_2^i$$
is symmetric as well.

We make $\cH$ into a $(\Z_{\geq 0}^I\times\Z)$-graded algebra as follows. For a polynomial $f\in \cH_{\gamma}$ of degree $k$
we define its bigrading to be $(\gamma,2k+\chi_Q(\gamma,\gamma)).$ It follows from either \eqref{Euler} or Theorem
\ref{explicit} that the product on $\cH$ is compatible with this bigrading.
Define the super-structure on $\cH$ to be induced by $\Z$-grading.

For two elements $a_{\gamma,k}\in\cH_{\gamma,k},$ $a_{\gamma',k'}\in\cH_{\gamma',k'},$ we have
$$a_{\gamma,k}a_{\gamma',k'}=(-1)^{\chi_Q(\gamma,\gamma')}a_{\gamma',k'}a_{\gamma,k}.$$
In general, this does not mean that $\cH$ is super-commutative. However, it is easy to twist the product by a sign, so that $\cH$ becomes super-commutative.
This can be done as follows.

Define the homomorphism of abelian groups $\epsilon:\Z^I\to \Z/2\Z$ by the formula
$$\epsilon(\gamma)=\chi_Q(\gamma,\gamma)\text{ mod }2.$$
Note that the parity of the element $a_{\gamma,k}$ equals to $\epsilon(\gamma)$ (by the definition).
We have a bilinear form
$$\Z^I\times\Z^I\to\Z/2\Z,\quad (\gamma_1,\gamma_2)\mapsto (\chi_Q(\gamma_1,\gamma_2)+\epsilon(\gamma_1)\epsilon(\gamma_2))\text{ mod }2,$$
which induces a symmetric form $\beta$ on the space $(\Z/2\Z)^I,$ such that $\beta(\gamma,\gamma)=0$ for all $\gamma\in (\Z/2\Z)^I.$
Hence, there exists a bilinear form $\psi$ on $(\Z/2\Z)^I$ such that
$$\psi(\gamma_1,\gamma_2)+\psi(\gamma_2,\gamma_1)=\beta(\gamma_1,\gamma_2).$$
Then the twisted product on $\cH$ is defined by the formula
$$a_{\gamma,k}\star a_{\gamma',k'}=(-1)^{\psi(\gamma,\gamma')}a_{\gamma,k}\cdot a_{\gamma',k'}.$$

It follows from the definition that the product $\star$ is associative, and the algebra $(\cH,\star)$ is super-commutative. From now on, we fix the choice
of bilinear form $\psi,$ and the corresponding product $\star$ on $\cH.$

\section{Freeness of COHA of a symmetric quiver}
\label{s:main_results}

\begin{theo}\label{free_algebra}For any finite symmetric quiver $Q,$ the $(\Z_{\geq 0}^I\times\Z)$-graded algebra $(\cH,\star)$
is a free super-commutative algebra generated by a $(\Z_{\geq 0}^I\times\Z)$-graded vector space $V$
of the form $V=V^{prim}\otimes\Q[x],$ where $x$ is a variable of bidegree $(0,2)\in\Z_{\geq 0}^I\times\Z,$
and for any $\gamma\in\Z_{\geq 0}^I$ the space $V_{\gamma,k}^{prim}$ is non-zero (and finite-dimensional) only for finitely many $k\in\Z.$\end{theo}

\begin{proof}Our first step is to construct the space $V.$ It will be convenient to treat $\cH_{\gamma}$ itself as a $\Z$-graded algebra (with the usual multiplication of polynomials, and the standard even grading). To distinguish the product in $\cH_{\gamma}$ and the product in $\cH,$ we will always
denote the last product by $"\star "$.

For convenience, we put $A_{\gamma}:=\Q[\{x_{i,\alpha}\}_{i\in I,1\leq \alpha\leq \gamma^i}]$ (considered as a $\Z$-graded algebra) and $P_{\gamma}:=\prod\limits_{i\in I}S_{\gamma^i}.$ Then
we have that $\cH_{\gamma}=A_{\gamma}^{P_{\gamma}}.$ Further, put
$$A_{\gamma}^{prim}:=\Q[(x_{j,\alpha_2}-x_{i,\alpha_1})_{i,j\in I,1\leq \alpha_1\leq \gamma^i,1\leq\alpha_2\leq \gamma^j}],\quad
\sigma_{\gamma}:=\sum\limits_{\substack{i\in I,\\ 1\leq \alpha\leq\gamma^i}}x_{i,\alpha}\in A_{\gamma}.$$
Then $A_{\gamma}=A_{\gamma}^{prim}\otimes\Q[\sigma_{\gamma}].$ Further, we have
$$\cH_{\gamma}=\cH_{\gamma}^{prim}\otimes\Q[\sigma_{\gamma}],\quad\cH_{\gamma}^{prim}:=(A_{\gamma}^{prim})^{P_{\gamma}}.$$

Now, for each $\gamma\in\Z_{\geq 0}^I,$ denote by $J_{\gamma}$ the smallest $P_{\gamma}$-stable $A_{\gamma}^{prim}$-submodule of the localization
$A_{\gamma}^{prim}[(x_{i,\alpha_2}-x_{i,\alpha_1})^{-1}_{i\in I,1\leq \alpha_1<\alpha_2\leq \gamma^i}],$ such that for all decompositions $\gamma=\gamma_1+\gamma_2,$ $\gamma_1,\gamma_2\in\Z_{\geq 0}^I\setminus \{0\},$ we have that
$$\frac{\prod\limits_{i,j\in I}\prod\limits_{\alpha_1=1}^{\gamma_1^i}\prod\limits_{\alpha_2=\gamma_1^j+1}^{\gamma^j}(x_{j,\alpha_2}-x_{i,\alpha_1})^{a_{ij}}}
{\prod\limits_{i\in I}\prod\limits_{\alpha_1=1}^{\gamma_1^i}\prod\limits_{\alpha_2=\gamma_1^i+1}^{\gamma^i}(x_{i,\alpha_2}-x_{i,\alpha_1})}\in J_{\gamma}.$$

Clearly, $J_{\gamma}^{P_{\gamma}}\subset\cH_{\gamma}^{prim}.$ Define $V_{\gamma}^{prim}\subset\cH_{\gamma}^{prim}$ to be a graded subspace
such that $$\cH_{\gamma}^{prim}=V_{\gamma}^{prim}\oplus J_{\gamma}^{P_{\gamma}}.$$
Further, put $$V_{\gamma}:=V_{\gamma}^{prim}\otimes\Q[\sigma_{\gamma}]\subset \cH_{\gamma},\quad V:=\bigoplus\limits_{\gamma\in\Z_{\geq 0}^I}V_{\gamma}.$$
We will prove that $V$ freely generates $\cH,$ and that all the spaces $V_{\gamma}^{prim}$ are finite-dimensional (this would imply the Theorem).

\begin{lemma}The subspace $V\subset\cH$ generates $\cH$ as an algebra.\end{lemma}

\begin{proof}Indeed, for each $\gamma\in\Z_{\geq 0}^I,$ the image of the multiplication map
$$\bigoplus\limits_{\substack{\gamma_1+\gamma_2=\gamma,\\ \gamma_1,\gamma_2\in\Z_{\geq 0}^I\setminus\{0\}}}\cH_{\gamma_1}\otimes\cH_{\gamma_2}\to\cH_{\gamma}$$
is precisely $J_{\gamma}^{P_{\gamma}}\otimes\Q[\sigma_{\gamma}]$ (this is straightforward to check). Hence, it follows by induction on $\sum\limits_{i\in I}\gamma^i$ that the subspace $\cH_{\gamma}$ is contained in the subalgebra generated by $V.$ This proves Lemma.\end{proof}

Now we will show that the spaces $V_{\gamma}^{prim}$ are finite-dimensional.

\begin{lemma}For each $\gamma\in\Z_{\geq 0}^I,$ the space $V_{\gamma}^{prim}$ is finite-dimensional.\end{lemma}

\begin{proof}In other words, we need to show that the ideal $J_{\gamma}^{P_{\gamma}}\subset\cH_{\gamma}^{prim}$
has finite codimension.
First note that if we replace $a_{ii}$ by $a_{ii}+1,$ then the fractional ideal $J_{\gamma}$ would become smaller or equal. Hence, we may and will assume,
that $a_{ii}>0$ for $i\in I,$ and so $J_{\gamma}\subset A_{\gamma}^{prim}.$

Since we have natural injective morphisms
$$\cH_{\gamma}^{prim}/J_{\gamma}^{P_{\gamma}}\hookrightarrow A_{\gamma}^{prim}/J_{\gamma},$$ it suffices to show that the ideal $J_{\gamma}\subset A_{\gamma}^{prim}$
has finite codimension. It will be convenient to treat the algebra $A_{\gamma}^{prim}$ as the algebra of functions
on the hyperplane $W\subset \A_{\Q}^{\sum\limits_{i\in I}\gamma^i},$ given by equation $\sigma_{\gamma}(x)=0.$

It suffices to show that
$$\supp(A_{\gamma}^{prim}/J_{\gamma})=\{0\}\subset W.$$
Assume the converse is true. Then there exists a point $y\in W_{\bar{\Q}},$ $y\ne 0,$ such that all the functions from $J_{\gamma}$
vanish at $y.$ Since $\sigma_{\gamma}(y)=0,$ we have that not all of the coordinates $y_{i,\alpha}$ are equal to each other. Since the ideal $J_{\gamma}$ is $P_{\gamma}$-stable, we may assume that there exists a decomposition $\gamma=\gamma_1+\gamma_2,$ $\gamma_1,\gamma_2\in\Z_{\geq 0}^I\setminus \{0\},$
such that
$$y_{i,\alpha_1}\ne y_{j,\alpha_2}\text{ for }1\leq \alpha_1\leq\gamma_1^i,\, \gamma_1^j+1\leq \alpha_2\leq \gamma^j.$$
But then the function
$$\frac{\prod\limits_{i,j\in I}\prod\limits_{\alpha_1=1}^{\gamma_1^i}\prod\limits_{\alpha_2=\gamma_1^j+1}^{\gamma^j}(x_{j,\alpha_2}-x_{i,\alpha_1})^{a_{ij}}}
{\prod\limits_{i\in I}\prod\limits_{\alpha_1=1}^{\gamma_1^i}\prod\limits_{\alpha_2=\gamma_1^i+1}^{\gamma^i}(x_{i,\alpha_2}-x_{i,\alpha_1})}\in J_{\gamma}$$
does not vanish at $y,$ a contradiction.

Lemma is proved.
\end{proof}

It remains to prove the freeness.

\begin{lemma}The subspace $V\subset \cH$ freely generates $\cH.$\end{lemma}

\begin{proof}We have already shown the generation. So we need to show freeness.

Choose an order on $I,$ and fix the corresponding lexicographical order on $\Z_{\geq 0}^I$ (denoted by $\gamma\succeq\gamma'$).
Further, denote by $e_{\gamma,\beta},$ $1\leq \beta\leq \dim V_{\gamma}^{prim},$ a homogeneous basis of $V_{\gamma}^{prim}.$
We have the lexicographical order on all of the elements $e_{\gamma,\beta}$ (for all $\gamma$ and $\beta$). Further, the elements $e_{\gamma,\beta}\sigma_{\gamma}^m$ (for all $\gamma,\beta,m$) form the basis of $V,$ and again we have a lexicographical order on them, which we
denote by $\succeq.$

Fix some $\gamma\in\Z_{\geq 0}^I.$ Consider the set $Seq_{\gamma}$ of all {\it nonincreasing} sequences $(e_{\gamma_1,\beta_1}\sigma_{\gamma_1}^{m_1},\dots,e_{\gamma_d,\beta_d}\sigma_{\gamma_d}^{m_d})$ such that

1) $\gamma_1+\dots+\gamma_d=\gamma;$

2) an equality $(\gamma_i,\beta_i,m_i)=(\gamma_{i+1},\beta_{i+1},m_{i+1})$ implies $\epsilon(\gamma_i)=0.$

Clearly, we have natural lexicographical order on $Seq_{\gamma}$ (which we again denote by $\succeq$). For a sequence $t\in Seq_{\gamma},$
we denote by $M_t\in\cH_{\gamma}$ the corresponding product.

What we need is to show non-vanishing of each non-trivial linear combination:
\begin{equation}\label{lin_comb}T=\sum\limits_{i=1}^n\lambda_iM_{t_i}\ne 0,\quad t_1,\dots,t_n\in Seq_{\gamma},\, t_1\succ\dots\succ t_n,\, \lambda_1\dots\lambda_n\ne 0.\end{equation}

Fix some $t_i$ and $\lambda_i$ as in \eqref{lin_comb}. Denote by $(\gamma_1,\dots,\gamma_k)$ the underlying sequence of elements in $\Z_{\geq 0}^I$ for the sequence $t_1\in Seq_{\gamma}.$ Then $\gamma_1+\dots+\gamma_k=\gamma,$ and $\gamma_i\ne 0.$ We have natural isomorphism
$$A_{\gamma}\cong A_{\gamma_1}\otimes\dots\otimes A_{\gamma_k}=:\widetilde{A_{\gamma}},$$ which induces an inclusion
$$\iota:\cH_{\gamma}\hookrightarrow \cH_{\gamma_1}\otimes\dots\otimes\cH_{\gamma_k}=:\widetilde{\cH_{\gamma}}.$$
Put $\widetilde{P_{\gamma}}:=P_{\gamma_1}\times\dots\times P_{\gamma_k}.$ Then we have $\widetilde{\cH_{\gamma}}=\widetilde{A_{\gamma}}^{\widetilde{P_{\gamma}}}.$ Further, take an ideal
$$(J_{\gamma_1} \cap A_{\gamma_1}^{prim}) \widetilde{A_{\gamma}}+\dots+(J_{\gamma_k}\cap A_{\gamma_k}^{prim})\widetilde{A_{\gamma}}=:\widetilde{J_{\gamma}}\subset \widetilde{A_{\gamma}}.$$
We will write $x^{(p)}_{i,\alpha}\in\widetilde{A_{\gamma}}$ for variables from the $p$-th factor $A_{\gamma_p}\subset \widetilde{A_{\gamma}}.$

{\noindent {\bf Claim.}} {\it The elements $(x^{(q)}_{j,\alpha_2}-x^{(p)}_{i,\alpha_1})\in \widetilde{A_{\gamma}},$
$1\leq p<q\leq k,$ are not zero divisors in the quotient ring
$$\widetilde{A_{\gamma}}/\widetilde{J_{\gamma}}.$$}
\begin{proof}For convenience, we may assume that the sequence $\gamma_1,\dots,\gamma_k$ is not necessarily non-increasing, and $q=k.$
Any element $g\in \widetilde{A_{\gamma}}$ can be written (in a unique way) as a sum
$$g=\sum\limits_{\nu=0}^N g_{\nu}\sigma_{\gamma_k}^{\nu},\quad g_{\nu}\in A_{\gamma_1}\otimes\dots\otimes A_{\gamma_{k-1}}\otimes A_{\gamma_k}^{prim}.$$

The following are obviously equivalent:

$(i)$ $g\not\in \widetilde{J_{\gamma}};$

$(ii)$ for some $\nu\in\{0,\dots,N\},$ $g_{\nu}\not\in\widetilde{J_{\gamma}}.$

Now suppose that $g\not\in\widetilde{J_{\gamma}}.$ We need to show that
\begin{equation}\label{not_in_ideal2}(x^{(k)}_{j,\alpha_2}-x^{(p)}_{i,\alpha_1})g\not\in\widetilde{J_{\gamma}}.\end{equation}
We may assume that $g_N\not\in \widetilde{J_{\gamma}}.$ Put
$$x^{(k)}_{av}:=\frac1{\sum\limits_{i\in I}\gamma_k^i}\sum\limits_{i,\alpha}x^{(k)}_{i,\alpha}=\frac1{\sum\limits_{i\in I}\gamma_k^i}\sigma_{\gamma_k}.$$
Then $x^{(k)}_{j,\alpha_2}-x^{(k)}_{av}\in A_{\gamma_k}^{prim},$ and we have
$$(x^{(k)}_{j,\alpha_2}-x^{(p)}_{i,\alpha_1})g=(x^{(k)}_{j,\alpha_2}-x^{(k)}_{av}-x^{(p)}_{i,\alpha_1})g+x^{(k)}_{av}g=\frac1{\sum\limits_{i\in I}\gamma_k^i}g_N\sigma_{\gamma_k}^{N+1}+
\sum\limits_{\nu=0}^Ng_{\nu}'\sigma_{\gamma_k}^{\nu}$$
for some $g_{\nu}'\in A_{\gamma_1}\otimes\dots\otimes A_{\gamma_{k-1}}\otimes A_{\gamma_k}^{prim}.$
Since $\frac1{\sum\limits_{i\in I}\gamma_k^i}g_N\not\in\widetilde{J_{\gamma}}$ by our assumption, this implies \eqref{not_in_ideal2}. Claim is proved.
\end{proof}

We put
$$\widetilde{A_{\gamma}}':=\widetilde{A_{\gamma}}[(x^{(q)}_{j,\alpha_2}-x^{(p)}_{i,\alpha_1})^{-1}_{1\leq p<q\leq k}],\quad \widetilde{\cH_{\gamma}}':=
(\widetilde{A_{\gamma}}')^{\widetilde{P_{\gamma}}}.$$
We denote by the same letter $L$ the localization maps $L:\widetilde{A_{\gamma}}\to \widetilde{A_{\gamma}}',$ $L:\widetilde{\cH_{\gamma}}\to \widetilde{\cH_{\gamma}}'.$ Also put $\widetilde{J_{\gamma}}':=\widetilde{A_{\gamma}}'L(\widetilde{J_{\gamma}}).$
It follows directly from Claim that the induced maps
\begin{equation}\label{injective}L:\widetilde{A_{\gamma}}/\widetilde{J_{\gamma}}\to \widetilde{A_{\gamma}}'/\widetilde{J_{\gamma}}',\quad
L:\widetilde{\cH_{\gamma}}/(\widetilde{J_{\gamma}})^{\widetilde{P_{\gamma}}}\to \widetilde{\cH_{\gamma}}'/(\widetilde{J_{\gamma}}')^P\end{equation}
are injective.

Now, let $r\in\{1,\dots,n\}$ be the maximal number such that the underlying sequence of elements in $\Z_{\geq 0}^I$ for $t_r$
coincides with $(\gamma_1,\dots,\gamma_k).$ Then it is straightforward to check that
$$L\iota(M_{t_l})\in (\widetilde{J_{\gamma}}')^{\widetilde{P_{\gamma}}}\text{ for }r+1\leq l\leq n.$$

Thus, it suffices to show that
\begin{equation}\label{not_in_ideal}L\iota(\sum\limits_{i=1}^r\lambda_iM_{t_i})\not\in(\widetilde{J_{\gamma}}')^{\widetilde{P_{\gamma}}}.\end{equation}
For all relevant $\beta_i,m_i$ we have the following comparison:
\begin{multline}\label{ideal1}L\iota(e_{\gamma_1,\beta_1}\sigma_{\gamma_1}^{m_1}\star\dots\star e_{\gamma_k,\beta_k}\sigma_{\gamma_k}^{m_k})\equiv\\
F_{\gamma_1,\dots,\gamma_k}\cdot \sum\limits_{\tau}s(\tau)e_{\gamma_{1},\beta_{\tau(1)}}\sigma_{\gamma_{1}}^{m_{\tau(1)}}\otimes\dots\otimes e_{\gamma_{k},\beta_{\tau(k)}}\sigma_{\gamma_{k}}^{m_{\tau(k)}}\text{ mod }(\widetilde{J_{\gamma}}')^{\widetilde{P_{\gamma}}},\end{multline}
where the sum is taken over all permutations $\tau\in S_k$ such that $\gamma_p=\gamma_{\tau(p)}$ for all $p\in\{1,\dots,k\},$
and $s(\tau)$ is the Koszul sign (recall that the parity of $e_{\gamma,\beta}\sigma_{\gamma}^k$ equals to $\epsilon(\gamma)$),
and $F_{\gamma_1,\dots,\gamma_k}\in\widetilde{\cH_{\gamma}}'$ is (up to sign) the product of some powers (positive and ($-1$)-st) of the differences
$$(x^{(q)}_{j,\alpha_2}-x^{(p)}_{i,\alpha_1})\in \widetilde{A_{\gamma}},\quad 1\leq p<q\leq k.$$
Thus, $F_{\gamma_1,\dots,\gamma_k}$ is invertible, and according to \eqref{ideal1} and injectivity of maps \eqref{injective}, we are left to check that
$$\sum\limits_{\tau}s(\tau)e_{\gamma_{1},\beta_{\tau(1)}}\sigma_{\gamma_{1}}^{m_{\tau(1)}}\otimes\dots\otimes e_{\gamma_{k},\beta_{\tau(k)}}\sigma_{\gamma_{k}}^{m_{\tau(k)}}\not\in \widetilde{J_{\gamma}}^{\widetilde{P_{\gamma}}}.$$
But this follows from the condition 2) in the above definition of the set of sequences $Seq_{\gamma},$ and from the definition of $e_{\gamma_i,\beta}.$
This proves \eqref{not_in_ideal2}, hence the desired linear independence \eqref{lin_comb}, and hence free generation. Lemma is proved.
\end{proof}
Theorem is proved.
\end{proof}

It is clear that if $V_{\gamma,k}^{prim}\ne 0$ in the notation of the above Theorem, then $k\equiv \chi_Q(\gamma,\gamma)\text{ mod }2$ and $k\geq
\chi_Q(\gamma,\gamma).$ Our next result is an upper bound on $k$ (depending on $\gamma$) for which $V_{\gamma,k}\ne 0.$

For a given symmetric quiver $Q$ and $\gamma\in\Z_{\geq 0}^I\setminus\{0\},$ we put
$$N_{\gamma}(Q):=\frac12(\sum\limits_{\substack{i,j\in I,\\ i\ne j}}a_{ij}\gamma^i\gamma^j+\sum\limits_{i\in I}\max(a_{ii}-1,0)\gamma^i(\gamma^i-1))-\sum\limits_{i\in I}\gamma^i+2.$$

\begin{theo}\label{upper_bound}In the notation of Theorem \ref{free_algebra}, if $V_{\gamma,k}^{prim}\ne 0,$ then $\gamma\ne 0,$
$$k\equiv \chi_Q(\gamma,\gamma)\text{ mod }2,\quad\text{and}\quad \chi_Q(\gamma,\gamma)\leq k< \chi_Q(\gamma,\gamma)+2N_{\gamma}(Q).$$\end{theo}

\begin{proof}According to the proof of Theorem \ref{free_algebra}, we have \begin{equation}\label{equal_dim}\dim V_{\gamma,k}^{prim}=\dim (\cH_{\gamma}^{prim}/J_{\gamma}^{P_{\gamma}})^{k-\chi_Q(\gamma,\gamma)}.\end{equation}
Recall that $P_{\gamma}=\prod\limits_{i\in I}S_{\gamma^i},$
$$A_{\gamma}^{prim}:=\Q[(x_{j,\alpha_2}-x_{i,\alpha_1})_{i,j\in I,1\leq \alpha_1\leq \gamma^i,1\leq\alpha_2\leq \gamma^j}],\quad
\cH_{\gamma}^{prim}:=(A_{\gamma}^{prim})^{P_{\gamma}},$$
and $J_{\gamma}$ is the smallest $P_{\gamma}$-stable $A_{\gamma}^{prim}$-submodule of the localization
$$A_{\gamma}^{prim}[(x_{i,\alpha_2}-x_{i,\alpha_1})^{-1}_{i\in I,1\leq \alpha_1<\alpha_2\leq \gamma^i}],$$ such that for all decompositions $\gamma=\gamma_1+\gamma_2,$ $\gamma_1,\gamma_2\in\Z_{\geq 0}^I\setminus \{0\},$ we have that
\begin{equation}\label{expr}\frac{\prod\limits_{i,j\in I}\prod\limits_{\alpha_1=1}^{\gamma_1^i}\prod\limits_{\alpha_2=\gamma_1^j+1}^{\gamma^j}(x_{j,\alpha_2}-x_{i,\alpha_1})^{a_{ij}}}
{\prod\limits_{i\in I}\prod\limits_{\alpha_1=1}^{\gamma_1^i}\prod\limits_{\alpha_2=\gamma_1^i+1}^{\gamma^i}(x_{i,\alpha_2}-x_{i,\alpha_1})}\in J_{\gamma}.\end{equation}
Recall that we take the standard even grading on $A_{\gamma}^{prim}$ with $\deg(x_{j,\alpha_2}-x_{i,\alpha_1})=2,$ and the induced grading on $\cH_{\gamma}^{prim}.$

According to \eqref{equal_dim}, it suffices to prove inclusions
\begin{equation}\label{ideal_generates}(A_{\gamma}^{prim})^d\subset J_{\gamma}\text{ for } d\geq 2N(Q).\end{equation}
For any $i,j\in I,$ put $$a_{ij}':=\begin{cases}a_{ij} & \text{ if }i\ne j;\\
\max(1,a_{ii}) & \text{ if }i=j.\end{cases}$$
Take the quiver $Q':=(I,a_{ij}').$ Note that $N_{\gamma}(Q)=N_{\gamma}(Q'),$ and if we replace $Q$ by $Q',$ then the new fractional $J_{\gamma}$ will be contained in the initial one.
Hence, in order to prove inclusions \eqref{ideal_generates}, we may and will assume that $a_{ii}\geq 1$ for $i\in I,$ and so $J_{\gamma}\subset A_{\gamma}^{prim}.$ We will deduce \eqref{ideal_generates} from the following more general result.

\begin{lemma}\label{effective_gen}Let $\mathrm{k}$ be arbitrary field, and consider the graded algebra of polynomials $B=\mathrm{k}[z_1,\dots,z_n],$ $n\geq 1,$ with grading $\deg(z_i)=1.$ Suppose that
$l_1,\dots,l_s\in B^1$ are pairwise linearly independent non-zero linear forms in $z_i.$ Take some non-empty set of polynomials $\{P_1,\dots,P_r\}\subset B$ of the form
$$P_i=l_1^{d_{i1}}\dots l_s^{d_{is}},$$
where $d_{ij}\in\Z_{\geq 0}.$ Put $d_j:=\max_{1\leq i\leq r}d_{ij},$ $1\leq j\leq s.$ Then the following are equivalent:

(i) $B^d\subset (P_1,\dots,P_r)$ for $d\geq d_1+\dots+d_s-n+1;$

(ii) the ideal $(P_1,\dots,P_r)\subset B$ has finite codimension;

(iii) For any sequence $p_1,\dots,p_r$ of numbers in $\{1,\dots,s\},$ such that $d_{i,p_i}>0$ for $1\leq i\leq r,$ the linear forms $l_{p_1},\dots,l_{p_r}$ generate the space $B^1.$\end{lemma}

\begin{proof}Both implications $(i)\Rightarrow (ii)$ and $(ii)\Rightarrow (iii)$ are evident. So we are left to prove implication
$(iii)\Rightarrow (i).$

Put $D:=d_1+\dots+d_s-n+1.$ If $D\leq 0,$ then one of the polynomials $P_i$ is constant, and there is nothing to prove. So, we assume that $D>0.$

We proceed by induction on $D+n.$ If $D+n=2,$ then $n=s=d_1=D=1,$ hence $(P_1,\dots,P_r)\supset (z_1),$ and the statement is proved.

Assume that the implication holds for  $D+n<k_0>2.$ We will prove that it holds for $D+n=k_0.$ Consider the cases.

{\it The case 0.} One of $P_i$ is constant. Then, there is nothing to prove.

{\it The case 1.} We have $P_i=l_j$ for some $i,j.$ Then it suffices to show that the images of $P_{i^\prime}$ with $d_{i'j}=0$ in $B/(l_j)$ generate
$(B/(l_j))^d$ for $d\geq D.$ If $n=1,$ then this is clear, and if $n>1,$ then this follows from the induction hypothesis.

{\it The case 2.} All $P_i$ have degree at least $2.$ Take $d\geq D,$ and $f\in B^d.$ Choose some sequence $p_1,\dots,p_r$ of numbers in $\{1,\dots,s\},$ such that $d_{i,p_i}>0$ for $1\leq i\leq r.$ Then by $(iii)$ we can write
$$f=\sum\limits_{i=1}^rl_{p_i}g_i,\quad g_i\in B^{d-1}.$$
It suffices to show that for each $1\leq i\leq r,$ the polynomial $g_i$ belongs to an ideal generated by $P_{i^\prime}$ with $l_{p_i}\nmid P_{i^\prime},$
and $\frac{P_{i^{\prime\prime}}}{l_{p_i}}$ with $l_{p_i}\mid P_{i^{\prime\prime}}.$ But this follows from induction hypothesis.

In each case, we have proved the desired implication. Induction statement is proved. Lemma is proved.
\end{proof}

Now, consider the cases. If $\sum\limits_{i}\gamma^i=1,$ then $N_{\gamma}(Q)=1,$ and $A_{\gamma}^{prim}=\Q,$ hence inclusions \eqref{ideal_generates}
hold. Further, if $\sum\limits_{i}\gamma^i\geq 2,$ then we apply Lemma \ref{effective_gen} to $B=A_{\gamma}^{prim},$ the linear forms $(x_{j,\alpha_2}-x_{i,\alpha_1})$
(defined up to sign), and polynomials which are in the $P_{\gamma}$-orbit of the expressions \eqref{expr}. They generate precisely the ideal $J_{\gamma}\subset A_{\gamma}^{prim}.$ We have already shown in the proof of Theorem \ref{free_algebra} that the ideal $J_{\gamma}\subset A_{\gamma}^{prim}$ has finite codimension. Therefore, implication $(ii)\Rightarrow (i)$ from Lemma \ref{effective_gen} gives the desired inclusions \eqref{ideal_generates}. Indeed,
we have that $$d_1+\dots +d_s=\frac12(\sum\limits_{\substack{i,j\in I,\\ i\ne j}}a_{ij}\gamma^i\gamma^j+\sum\limits_{i\in I}(a_{ii}-1)\gamma^i(\gamma^i-1)),
\quad n=\sum\limits_{i\in I}\gamma^i-1,$$
hence $N_{\gamma}(Q)=d_1+\dots+d_s-n+1.$ The inclusions \eqref{ideal_generates} and Theorem are proved.
\end{proof}

\section{Applications to quantum DT invariants}
\label{s:DT_invariants}

Define the generating function for the COHA $\cH$ of symmetric quiver $Q$ by the following formula:
$$H_Q(\{t_i\}_{i\in I},q):=\sum\limits_{\gamma\in\Z_{\geq 0}^I,k\in\Z} (-1)^k\dim(\cH_{\gamma,k})t^{\gamma}q^{\frac{k}2}\in \Z((q^{\frac12}))[[\{t_i\}_{i\in I}]],$$
where $t^{\gamma}:=\prod\limits_{i\in I}t_i^{\gamma^i}.$ Note that we have an equality
\begin{equation}\label{stupid}H_Q=\sum\limits_{\gamma\in\Z_{\geq 0}^I}\frac{(-q^{\frac12})^{\chi_Q(\gamma,\gamma)}}{\prod\limits_{i\in I}(1-q)(1-q^2)\dots(1-q^{\gamma^i})}t^{\gamma}.\end{equation}

Recall the notation
$$(z;q)_{\infty}:=\prod\limits_{n\in\Z_{\geq 0}}(1-q^nz)$$
(the so-called $q$-Pochhammer symbol).

\begin{cor}\label{positive}Let $Q$ be a symmetric quiver. Then we have a decomposition
$$H_Q(\{t_i\}_{i\in I},q)=\prod\limits_{\gamma\in\Z_{\geq 0}^I,k\in\Z}(q^{\frac{k}2}x^{\gamma};q)_{\infty}^{(-1)^{k-1}c_{\gamma,k}},$$
where $c_{\gamma,k}$ are non-negative integer numbers.
Moreover, if $c_{\gamma,k}\ne 0,$ then $\gamma\ne 0,$ $$k\equiv \chi_Q(\gamma,\gamma)\text{ mod }2,\quad\text{and}\quad \chi_Q(\gamma,\gamma)\leq k< \chi_Q(\gamma,\gamma)+2N_{\gamma}(Q).$$
In particular, for a fixed $\gamma$ only finitely many of $c_{\gamma,k}$ are non-zero.\end{cor}

\begin{proof}Corollary follows immediately from Theorem \ref{free_algebra} and Theorem \ref{upper_bound} if we put $c_{\gamma,k}=\dim V_{\gamma,k}^{prim}.$
Indeed, the generating function of the free super-commutative subalgebra generated by one element of bidegree $(\gamma,k)$ equals to
$$(1-q^{\frac{k}2}t^{\gamma})^{(-1)^{k-1}}.$$ The resulting decomposition follows from free generation of $\cH$ by $V,$ and from Theorem \ref{upper_bound}.\end{proof}

In the notation of Corollary \ref{positive} and the terminology of \cite{KS}, the polynomials
$$\Omega(\gamma)(q):=\sum\limits_{k\in\Z}c_{\gamma,k}q^{\frac{k}2}\in\Z[q^{\pm \frac12}]$$
are quantum Donaldson-Thomas invariants of the quiver $Q$ with trivial potential, stability function with unique slope, and the dimension vector $\gamma.$
It follows from Corollary \ref{positive} that for $\gamma\ne 0$ we have
$$\Omega(\gamma)(q)=q^{\frac12 \chi_Q(\gamma,\gamma)} \tilde{\Omega}(\gamma)(q),$$
where $\tilde{\Omega}(\gamma)(q)$ is a polynomial with non-negative coefficients, $\tilde{\Omega}(\gamma)(0)=1,$ and $\deg(\tilde{\Omega}(\gamma)(q))<N_{\gamma}(Q).$

We would like to mention a connection with the paper of Reineke \cite{R}. In loc. cit., for each integer $m\geq 1,$ the following $q$-hypergeometric series is considered:
$$H(q,t)=H_m(q,t):=\sum\limits_{n\geq 0}\frac{q^{(m-1)\binom{n}2}}{(1-q^{-1})(1-q^{-2})\dots(1-q^{-n})}t^n\in\Z(q)[[t]].$$
Denote by $Q_m$ the $m$-loop quiver (a quiver with one vertex and $m$ loops). Since $\chi_{Q_m}(n_1,n_2)=(1-m)n_1n_2,$ the formula \eqref{stupid} implies
$$H_m(q,t)=H_{Q_m}((-1)^{m-1}tq^{\frac{1-m}2},q^{-1}).$$
Also, we have $N_n(Q_m)=(m-1)\binom{n}2-n+2.$ Therefore, Corollary \ref{positive} implies the following.

\begin{cor}\label{positive_special}
$$H_m(q,(-1)^{m-1}t)=\prod\limits_{n\geq 1,k\in\Z}(q^kt^n;q^{-1})^{-(-1)^{(m-1)n}d_{n,k}},$$
where $d_{n,k}$ are non-negative integers, and inequality $d_{n,k}>0$ implies $$n-1\leq k\leq (m-1)\binom{n}2.$$
In particular, for a fixed $n$ only finitely many of $d_{n,k}$ are non-zero.\end{cor}

This Corollary is stronger than Conjecture 3.3 in \cite{R}. According to notation of \cite{R}, the quantized Donaldson-Thomas type invariant
$DT_n^{(m)}(q)$ equals to $\sum\limits_{k\in\Z} d_{n,k}q^k.$ Thus, Corollary \ref{positive_special} implies that $DT_n^{(m)}(q)$ is a monic polynomial of degree $(m-1)\binom{n}2,$ divisible by $q^{n-1},$ with non-negative coefficients.

With above said, the numbers $d_{n,k}$ are dimensions of graded components of finite-dimensional graded algebras $\cH_n^{prim}/J_n^{S_n}.$ It would be interesting to compare this interpretation with explicit formulas for $DT_n^{(m)}(q)$ in \cite{R}, Theorem 6.8.




\end{document}